\documentclass[11pt]{article}
\usepackage{amsmath}
\usepackage{amsthm}
\usepackage{amssymb}
\usepackage{latexsym}
\newcommand{\RR}{\mathbb{R}}
\newcommand{\NN}{\mathbb{N}}

\newcommand{\cB}{\mathcal{B}}

\renewcommand{\epsilon}{\varepsilon}
\renewcommand{\phi}{\varphi}

\newcommand{\ER}{\overline{\mathbb{R}}}
\newtheorem{theorem}{Theorem} 
\newtheorem{lemma}{Lemma}
\newtheorem{proposition}{Proposition}

\begin{document} 

\title{Non-existence of polar factorisations and polar inclusion of a 
vector-valued mapping} 
\author{R.J. Douglas\thanks{Institute 
of Mathematical and Physical Sciences, 
Aberystwyth University, Aberystwyth SY23 3BZ,   
U.K. {\em E-mail address} {\tt rsd@aber.ac.uk}}}  

\date{}

\maketitle 

{\em To appear in the International Journal 
of Pure and Applied Mathematics}, (IJPAM), {\bf 41}, no. 3, (2007), 
363-374.

\begin{abstract}
This paper proves some results concerning the polar
factorisation of an integrable vector-valued function $u$ into the composition
$u = u^{\#} \circ s$, where $u^{\#} = \nabla \psi$ 
almost everywhere 
for some convex function $\psi$, 
and $s$ is a measure-preserving mapping.
Not every integrable function has a polar factorisation; we extend the
class of counterexamples.
We introduce a generalisation: $u$ has a polar inclusion if 
$u(x) \in \partial \psi (y)$ for almost every pair $(x,y)$ with 
respect to a measure-preserving plan. Given a regularity assumption,
we show that such measure-preserving plans are exactly the minimisers
of a Monge-Kantorovich optimisation problem. 
\end{abstract}

{\em Keywords}: Polar factorisation, Monotone rearrangement, 
Measure-preserving mappings, Monge-Kantorovich problem. 

{\em AMS classification}: 28A50, 28D05, 46E30

\section{Introduction}
\label{s:1}

A vector-valued function has a {\em polar factorisation} if it can be
written as the composition of its {\em monotone rearrangement}, which is
equal almost everywhere to the gradient of a convex function, with a
measure-preserving mapping. 
This concept was introduced by Brenier \cite{YB87,YB91}, 
who proved existence and uniqueness of the monotone rearrangement on
sufficiently regular domains, and existence and uniqueness of the polar
factorisation subject to a further ``nondegeneracy'' restriction on the
function.
Burton and Douglas \cite{BD98,BD03}  investigated the consequences of relaxing
his assumptions by studying polar factorisations on general sets of finite
Lebesgue measure.
While the monotone rearrangement continues to exist and be unique, as proved
by McCann \cite{RJM}, they demonstrated that there is
a class of functions  which have no polar 
factorisation (that is there is no measure-preserving mapping 
satisfying the definition). In this paper we extend this class of 
functions. A natural question to ask is whether a weaker version of the 
concept holds for functions which do not have a polar factorisation. 
We introduce {\em polar inclusion} of an integrable function, where the 
gradient of the convex function is replaced by the subdifferential, and
the inclusion is required to hold for almost every pair of points with 
respect to a product measure with prescribed marginals, that is a 
measure-preserving plan. It is easy to see that a polar factorisation 
induces a polar inclusion. We introduce a Monge-Kantorovich problem 
suitable for the study of polar inclusions. Given a regularity 
assumption on the cost function (which depends on the integrable
function), minimising measure-preserving plans are exactly those 
that arise in polar inclusions (of the integrable function).   

In this paper, given an integrable function $u : X \rightarrow \RR^n$, and
a set $Y \subset \RR^n$ of finite positive Lebesgue measure, we say that $u$
has a {\em polar factorisation through} $Y$ if $u = u^{\#} \circ s$, where
$u^{\#} = \nabla \psi$ almost everywhere
in $Y$ for a convex function $\psi$, 
and $s: X \rightarrow Y$ is a measure-preserving mapping. 
Existence and uniqueness of a polar factorisation refers to 
existence and uniqueness of the measure-preserving mapping $s$.
The restriction on $X$ is not severe; we only require that $(X, \mu)$ is a
complete measure space with the same measure-theoretic structure as an
interval of length $\mu (X)$ equipped with Lebesgue measure. 
(We give precise definitions below.) 
We say that $u$ has a {\em polar inclusion} through $Y$ if 
$u(x) \in \partial \psi (y)$ for $\pi$ almost every $(x,y)$, where 
$\partial \psi (y)$ denotes the subdifferential of $\psi$ at $y$, 
and $\pi$ is a measure on $X \times Y$ with prescribed marginals. 
If a polar factorisation exists for some measure-preserving mapping 
$s: X \rightarrow Y$, then the push-forward measure 
$(id \times s)_{\#}\mu$ is a measure-preserving plan for which 
$u$ has a polar inclusion.

For a given integrable $u$, previous work \cite{BD98,BD03} has 
established that there is uniqueness exactly when $u^{\#}$ is almost
injective (injective on a set of full measure) and general 
existence when $u^{\#}$ is almost injective on the complement 
of its level sets of positive size. Burton and Douglas 
\cite[Theorem 2]{BD03} established that if $u^{\#}$ does not 
satisfy the hypothesis of the general existence theorem, one can
construct a rearrangement $\hat{u}$ of $u^{\#}$ which is almost 
injective on the complement of its level sets of positive size, and
$\hat{u}$ does not have a polar factorisation. We extend this result in
Theorems 1 and 2. The key intermediate result is that for any set of 
full measure, we can find a level set of $u^{\#}$, not of positive 
size, which has uncountable intersection with the set of full measure. 
For any rearrangement ${\tilde u}$ of $u^{\#}$ satisfying 
the condition that all its level sets are countable except for the 
level sets of positive size, existence of a polar factorisation of
${\tilde u}$ would yield a contradiction. One way to think 
about the existence/uniqueness question is to consider 
if a measure-preserving mapping can 
be constructed that maps level sets of $u$ to corresponding level 
sets of $u^{\#}$. In the case of the uniqueness and general 
existence results, the non-empty level sets of $u^{\#}$ are 
singletons (ignoring level sets of positive size), therefore the
mapping is fixed. In the nonexistence results we would be 
mapping a countable set onto an uncountable set, which is a contradiction. 

Having extended the class of functions known to have no polar 
factorisation, we introduce a relaxed version of the concept,
polar inclusion. In an analogous manner to the generalisation 
of a Monge mass transfer problem to a Monge-Kantorovich problem,
we introduce measure-preserving plans, measures on $X \times Y$ 
with $\mu$ and $n$-dimensional Lebesgue measure $\lambda_n$ 
as their marginals. We study a Monge-Kantorovich problem with 
cost function $c(x,y) = |u(x) -y|^2/2$ in Theorem 3, and 
show that any measure-preserving
plan arising from a polar inclusion is a minimiser; conversely, if
$c$ is assumed to be lower semicontinuous, any minimiser yields 
a polar inclusion. This result can be viewed as a generalisation of
\cite[Corollary 1]{BD03}. If a minimiser is not supported on the graph 
of a single mapping, this yields a polar inclusion without a corresponding 
polar factorisation. Gangbo and McCann \cite{WGRJM} and Plakhov 
\cite{PLA} have studied Monge-Kantorovich problems for specific costs 
where the minimisers are of the type described above. It is a plausible
conjecture that polar inclusions exist for every integrable 
$u$. Such a result would fit well with the theory of geodesics 
in the space of measure-preserving mappings (see 
Brenier \cite{YB89,YB99}, Ambrosio and Figalli \cite{AF}) where the
problem cannot be solved in general within the set of measure-preserving 
mappings, it being necessary to consider the larger set of
measure-preserving plans. 

\subsection{Definitions and notation}

{\sc Definition}. Let $(X, \mu)$ and $(Y, \nu)$ 
be finite positive measure spaces with $\mu (X) = 
\nu (Y)$. Two vector-valued functions $f \in L^1 (X, \mu, \RR^n)$
and $g \in L^1 (Y, \nu, \RR^n)$ are {\em rearrangements} of
each other (or {\em equimeasurable}) if 
$$ \mu (f^{-1} (B)) = \nu (g^{-1}(B)) \mbox{ for every $B \in 
\cB (\RR^n)$,} $$
where $\cB (\RR^n)$ denotes the Borel field of $\RR^n$. 
Equivalent formulations can be found in Douglas \cite{RJD}.

\vspace{3mm}

\noindent {\sc Definitions}. A {\em measure-preserving mapping}
from a finite positive measure space $(X, \mu )$ to a 
positive measure space $(Y, \nu)$ with $\mu (X) = \nu (Y)$ 
is a mapping $s: X \rightarrow Y$ such that for each 
$\nu$-measurable set $A \subset Y$, $\mu (s^{-1} (A))= \nu (A)$.

We will be considering the special case of $(X, \mu)$ complete,
$Y \subset \RR^n$
and $\nu$ being $n$-dimensional Lebesgue measure. The 
$\nu$-measurable sets will be the Borel-measurable sets; 
the same measure-preserving properties can then be deduced  
for the Lebesgue-measurable sets.

\vspace{3mm}

\noindent
A finite complete measure space $(X, \mu)$ is a {\em measure-interval}
if there exists a measure-preserving transformation from $(X, \mu)$ to
$[0, \mu (X)]$ with Lebesgue measure (on the Lebesgue sets). 
(Measure-preserving transformations are defined in \cite{BD03}.)
We recall that any Polish space, that is any complete separable metric space, 
equipped with a finite
nonatomic Borel measure, is a measure interval.

\vspace{3mm}

Throughout this paper we will denote $n$-dimensional Lebesgue
measure by $\lambda_n$, and the extended real numbers, 
that is the set $\RR \cup \{-\infty, \infty\}$, by $\ER$. 

\vspace{3mm}

\noindent {\sc Definition}. Let $u \in L^1 (X, \mu, \RR^n)$,
where $(X, \mu)$ is a measure-interval. Let 
Lebesgue measurable $Y \subset \RR^n$ be such that 
$\lambda_n (Y) = \mu (X)$. The {\em monotone rearrangement
of $u$ on $Y$} is the unique function $u^{\#} : Y 
\rightarrow \RR^n$ that is 
a rearrangement of $u$, and satisfies $u^{\#} = \nabla \psi $
almost everywhere in $Y$ for some proper lower semicontinuous convex 
function $\psi : \RR^n \rightarrow \ER$.
(A $\ER$-valued function is called {\em proper} if it is not identically
$\infty$, and nowhere takes the value $- \infty$.)
  
The existence and uniqueness of the monotone rearrangement follows
from the main result of McCann \cite{RJM}. 
It is unique in the sense that if $\phi : \RR^n \rightarrow 
\ER$ is another convex function, and $\nabla \phi$ (as a function 
defined on $Y$) is a rearrangement of $u$, then $\nabla \phi (y) 
= \nabla \psi (y)$ for almost every $y \in Y$. 

\vspace{3mm}

\noindent {\sc Definition}.
Let $u \in L^1 (X, \mu, \RR^n)$ where $(X, \mu)$ is a measure-interval.
Let Lebesgue measurable $Y \subset \RR^n$ be such that $\lambda_n(Y)=\mu(X)$,
and let $u^{\#}$ denote the monotone rearrangement of $u$ on $Y$.
We say $u$ has a {\em polar factorisation through} $Y$ if there exists a
measure-preserving mapping $s$ from $(X,\mu)$ to $(Y,\lambda_n)$ such that
$u = u^{\#} \circ s$ almost everywhere. 

\vspace{3mm}

\noindent {\sc Definitions}. A mapping $f: U \rightarrow V$, where 
$(U, \mu)$ is a finite positive measure space, is {\em almost 
injective} if there exists a set $U_0 \subset U$ such that 
$f$ restricted to $U_0$ is injective, and $\mu (U \verb+\+ U_0)
= 0$. 

We say that $f$ is {\em almost 
countable to one} if there exists a set $U_0$ of full measure 
such that the intersection of any level set of $f$ with 
$U_0$ is countable. 

For $m \in \NN$, we say that 
$f$  is {\em almost m to 1} if there exists
a set $U_0$ of full measure
such that the intersection of any level set of $f$ with $U_0$ 
(whenever it is nonempty) has exactly $m$ elements.  
In this notation, almost injective functions are called almost one to one. 

\vspace{3mm}

\noindent {\sc Definition}. Let $(X, \mu)$ be a measure interval, and
let Lebesgue measurable $Y \subset \RR^n$ satisfy $\lambda_n (Y) = \mu (X)$. 
We say that 
a measure $\pi$ on $X \times Y$ is a {\em measure-preserving plan}
if $\mu (A) = \pi (A \times Y)$ and $\lambda_n (B) = \pi (X \times B)$
for all (respectively $\mu$ and $\lambda_n$) measurable subsets 
$A \subset X$, $B \subset Y$. 

We denote the set of measure-preserving plans by $\Pi (\mu, \lambda_n)$. 

\vspace{3mm}

\noindent {\sc Definition}. Let $u \in L^1 (X, \mu, \RR^n)$ where 
$(X, \mu)$ is a measure interval, 
and let Lebesgue measurable $Y \subset \RR^n$ satisfy 
$\lambda_n (Y) = \mu (X)$. We say that $u$ has a {\em polar  
inclusion through} $Y$ if there exists a (proper lower semicontinuous) 
convex function $\psi: \RR^n \rightarrow \ER$ and a measure-preserving
plan $\pi \in \Pi (\mu, \lambda_n)$ such that $u(x) \in \partial 
\psi (y)$ for $\pi$ a.e. $(x,y)$, where $\partial \psi (y)$ denotes 
the subdifferential of $\psi$ at $y$. 

\subsection{Statements of results}
Our main results are Theorems 1,2 and 3 below; the proofs of the first 
two are given in Section 2, the last in Section 3.

\begin{theorem}
\label{thm:1}
Let integrable $u^{\#}: Y \rightarrow 
\RR^n$ be the restriction of the gradient of a proper 
lower semicontinuous convex function to a set $Y \subset \RR^n$
of finite positive Lebesgue measure, and suppose that 
$u^{\#}$ restricted to the complement of its level sets 
of positive measure is not almost injective. Let 
$(X, \mu)$ be a measure-interval satisfying $\mu (X) =
\lambda_n (Y)$.  
Suppose 
$u : X \rightarrow \RR^n$ is a rearrangement of 
$u^{\#}$ which is {\em almost countable to 1} on the complement 
of its level sets of positive measure.  
Then $u$ does not have a polar 
factorisation (through $Y$).
\end{theorem}

\vspace{3mm}

The following extension of Theorem 1 yields further examples 
of integrable functions which do not have polar 
factorisations. 

\begin{theorem}
\label{thm:2}
Let $Y \subset \RR^n$ be a set 
of finite positive Lebesgue measure, and let $(X, \mu)$ be 
a measure-interval satisfying $\mu (X) = \lambda_n (Y)$.  
Let $u^{\#}$ be the monotone 
rearrangement through $Y \subset \RR^n$ of an integrable 
function $u: X \rightarrow \RR^n $, 
and suppose there is a Borel set $B \subset \RR^n$
such that 

\noindent (i) $u^{\#}$ restricted to $(u^{\#})^{-1}(B)$ is not almost 
injective on the complement of its level sets of positive measure; 

\noindent (ii) $u$ restricted to $u^{-1}(B)$ is almost 
countable to 1 on the complement of its level sets of positive 
measure. 

Then $u$ does not have a polar factorisation (through $Y$).  
\end{theorem}

\vspace{3mm}

\begin{theorem}
\label{thm:3}
Let $X$ be a Polish space, let $(X, \mu)$ be a measure-interval, 
and let $Y \subset \RR^n$ be a Lebesgue measurable set such that $\mu (X) 
= \lambda_n (Y)$ and $\int_Y |y|^2 d \lambda_n < \infty$. Let 
$u \in L^2 (X, \mu, \RR^n)$. Define $c: X \times Y \rightarrow 
[0, \infty) \cup \{ + \infty \}$ by $c(x,y) = |u(x) - y|^2/2$, 
and for $\gamma \in \Pi (\mu, \lambda_n)$ write 
$$ I (\gamma) = \int_{X \times Y} c(x,y) d \gamma (x,y). $$ 
\noindent (i) Suppose that $u(x) \in \partial \psi(y)$ for $\pi$ a.e. 
$(x,y)$, where $\pi \in \Pi (\mu, \lambda_n)$, and $\psi : \RR^n 
\rightarrow \ER$ is a convex lower semicontinuous function. Then  $\pi$ attains 
$\inf \{ I(\gamma) : \gamma \in \Pi(\mu, \lambda_n)\}$. 

\noindent (ii) Suppose that $c$ is lower semicontinuous, and let 
$\pi \in \Pi (\mu, \lambda_n)$ be a minimiser of $I (\gamma)$
over $\gamma \in \Pi (\mu, \lambda_n)$. Then there exists 
a convex lower semicontinuous function $\psi : \RR^n \rightarrow \ER$ 
such that 
$u(x) \in \partial \psi (y)$ for $\pi$ a.e. $(x,y)$. 
\end{theorem}

\section{Classes of functions without polar factorisations}
\label{s:2}

If a monotone rearrangement 
$u^{\#} = \nabla \psi$ (defined on 
$Y \subset \RR^n$)  is not almost injective on the 
complement of its level sets 
of positive measure, then it was proven in \cite[Theorem 2]{BD03}
that it
has a rearrangement $u$ which 
is almost injective on the complement of 
its level sets of positive measure,
and that $u$ does not have a polar factorisation (through $Y$). 
We extend this result in Theorem 1 to a larger class of 
functions (which are rearrangements of $u^{\#}$), those
that are almost countable to one on the complement of
their level sets of positive
measure. Moreover pairs $u^{\#}$, $u$ which satisfy these conditions 
(in a non-trivial way) 
for a restricted (Borel) set of values will also fail to have a 
polar factorisation, see Theorem 2. 
A key intermediate result, stated and proved in Lemma 1, is that a 
monotone rearrangement which is not almost injective on the
complement of its level sets
of positive size is not almost countable to one (on the 
complement of its level 
sets of positive size). 

\vspace{3mm}

We begin by
proving that every integrable $v : Y  \rightarrow \RR^n$ has 
a rearrangement $u$ which is almost $m$ to 1 on the
complement of its level sets
of positive measure. This demonstrates that Theorems \ref{thm:1} 
and \ref{thm:2} are a genuine extension of the results of 
\cite{BD03}. 
Other classes of almost countable to one functions 
can be constructed by analogous arguments. 

\vspace{3mm}

\begin{proposition} 
\label{prop:1}
Let $v : Y \rightarrow \RR^n$ be 
integrable, where $Y \subset \RR^n$ has finite Lebesgue measure.
Let $(X, \mu)$ be a measure-interval with 
$\mu (X) = \lambda_n (Y)$, and
let $m \in \NN$. Then $v$ has a rearrangement $u : X 
\rightarrow \RR^n$ which is almost $m$ to 1 on the complement 
of its level sets of positive measure.  
\end{proposition}

\begin{proof} 
We use the methods of \cite[Lemma 3]{BD03}. 
Initially we restrict attention to finding  
$\hat{u} : Y  \rightarrow \RR^n$, a rearrangement of 
$v$ which is almost $m$ to 1 on the complement of its
level sets of positive measure.  
Let $Y_i = v^{-1}(\alpha_i)$ for 
$i \in I$ be the level sets of $v$ that have positive measure,
where $I$ is a countable index set, and write $Y_0 = Y \verb+\+  
\cup_{i \in I} Y_i$. If $Y_0$ has zero measure, all rearrangements 
have the desired property. Otherwise, define a Borel measure 
$\nu$ on $\RR^n$ by $\nu (B) = \lambda_n (v^{-1}(B))$ for 
Borel sets $B \subset \RR^n$. Now $\{ \alpha_i : i \in I\}$ is
the set of atoms of $\nu$. Let $\nu_0$ be the nonatomic part
of $\nu$. 
Possibly adding or subtracting sets of zero size, we can
partition $Y_0$ into $m$ disjoint $G_{\delta}$ sets   
of equal $\lambda_n$-measure: we label these 
sets $Y^j_0$ for $j =1,...,m$.  
Now for $j=1,...,m$,
$(Y_0^j, \lambda_n)$ and $(\RR^n \verb+\+ 
\{\alpha_i : i \in I\}, \nu_0/m)$ are isomorphic
measure spaces (in the sense of \cite[Definition 2.2]{BD98}). 
This follows from  \cite[p. 164, Proposition 33, and 
p. 409, Theorem 16]{HLR} and  
\cite[Lemma 2.3]{BD98}. For each $j =1,...,m$, choose a 
measure-preserving bijection $u_j :Y_0^j \rightarrow 
\RR^n \verb+\+ \{\alpha_i: i \in I\}$. Define    
$u_0 : Y_0 \rightarrow \RR^n \verb+\+ \{\alpha_i : i \in I\}$
by $u_0 = u_j$ on $Y^j_0$ for each $j = 1,...,m$. Now $u_0$ is
(almost) $m$ to 1 by construction. Writing $v_0$ for $v$ 
restricted to $Y_0$, we have that for Borel $B \subset \RR^n$,
$$ \lambda_n (u_0^{-1} (B)) = \sum_{j=1}^m \lambda_n (u_j^{-1}(B))
= \frac{1}{m} \sum_{j=1}^m \nu_0 (B) = \nu_0 (B) = \lambda_n 
(v_0^{-1}(B)). $$
It follows that $u_0$ and $v_0$ are rearrangements. Define $\hat{u} = 
u_0$ on $Y_0$, and $\hat{u} = \alpha_i$ on $Y_i$ for $i \in I$. Then 
$\hat{u}: Y 
\rightarrow \RR^n$ is a rearrangement of $v$ having the desired properties. 

Finally we note that $(X, \mu)$ and $(Y, \lambda_n)$ are isomorphic, 
so we can choose a measure-preserving transformation 
$\tau : X \rightarrow Y$. We can choose sets $\tilde{X} \subset X$, 
$\tilde{Y} \subset Y$ of full measure such that $\tau: \tilde{X} 
\rightarrow \tilde{Y}$ is a bijection. The above construction 
yields $\hat{u} : \tilde{Y} \rightarrow \RR^n$ which is almost 
$m$ to 1 on the complement of its level sets of positive measure.   
Now $u: X \rightarrow \RR^n$ 
defined by $u = \hat{u} \circ \tau$ satisfies the required 
conditions. 
\end{proof}

\vspace{3mm}

\noindent {\em Notation}. We say that a measure $\mu$ is 
absolutely continuous with respect to a measure $\nu$, and write 
$\mu \ll \nu$, if $\mu(E) =0$ for every $\nu$-measurable 
set $E$ for which $\nu (E) =0$.

\vspace{3mm}

\begin{lemma}
\label{lemma:1}
Let integrable $u^{\#} : Y \rightarrow
\RR^n$ be the restriction of the gradient of a proper 
lower semicontinuous convex function to a set $Y \subset 
\RR^n$ of finite positive Lebesgue measure, and suppose
that $u^{\#}$ restricted to the complement of
its level sets of positive measure is not almost injective. 
Let $Y_0$ be the complement of the level sets of 
$u^{\#}$ of positive measure. 
Then given a set $\tilde{Y} \subset Y_0$ of full measure, 
we can find a level set 
of $u^{\#}$ of zero measure  
which has uncountable intersection with $\tilde{Y}$. 
\end{lemma}

\begin{proof}
We can choose a set $Y_1 \subset 
\tilde{Y}$ of positive measure such that for some integer 
$1 \leq k \leq n-1$, every point of $Y_1$ lies in a 
$k$-dimensional convex set that is a level set of $u^{\#}$. 
Then, as in the proof of Burton and Douglas \cite[Theorem 4]{BD03},
there is a homeomorphism $f$ from a compact set $Y_2 \subset Y_1$ 
with $\lambda_n (Y_2) > 0$, to a subset $f(Y_2)$ of 
$X \times \RR^k$, where $X$ is a Borel subset of $\RR^{n+1}$;
moreover the push-forward $\nu$ of $\lambda_n$ through $f$ satisfies
$\nu \ll \Gamma \times \lambda_k$, where $\Gamma$ is the finite 
Borel measure on $X$ described in Burton and Douglas  
\cite[Theorem 4]{BD03}. 

Now $\Gamma \times \lambda_k (f(Y_2))> 0$ by absolute 
continuity since $\nu (f(Y_2)) = \lambda_n (Y_2) > 0$,
and 
$$ \Gamma \times \lambda_k (f(Y_2)) = \int_X\int_{\RR^{k}}
1_{f(Y_2)} (g,l) d \lambda_k (l) d \Gamma (g), $$
so $\lambda_k (f(Y_2) \bigcap (\{g\} \times \RR^k)) > 0$ 
for some $g \in X$, thus $f(Y_2) \bigcap (\{g\} \times 
\RR^k)$ is uncountable. It follows that 
$Y_2 \bigcap f^{-1} (\{g\} \times \RR^k)$ is an uncountable
subset of a level set of $u^{\#}$. 
\end{proof}

\vspace{3mm}   

\noindent {\bf Proof of Theorem \ref{thm:1}}. 

\noindent  Suppose (for a contradiction) that $u$ has a
polar factorisation $u = u^{\#} \circ s$. Write $Y_0$ and $X_0$ for the 
complement of the level sets of positive measure of $u^{\#}$ and 
$u$ respectively. Modifying $s$ on a set of measure zero if 
necessary, we have that $u = u^{\#} \circ s$ where $s : X_0 
\rightarrow Y_0$ is measure-preserving.  
Choose a set 
${\tilde X} \subset X_0$ of full measure such that 

\noindent (i) $u(x) = u^{\#} \circ s (x)$ for every $x \in 
{\tilde X}$, and 

\noindent (ii) the intersection of every level set of $u$ with 
$\tilde{X}$ is countable. 

Now $s (\tilde{X}) \subset Y_0$ is a set of full measure;  
Lemma 1 yields the existence of a level set of $u^{\#}$ 
which has uncountable intersection with $s (\tilde{X})$. 
Write this intersection $\{ y_i \}_{i \in I}$ for some
uncountable index set $I$, where $y_i$ are distinct.  
For each $i$ we can choose $\beta_i \in \tilde{X}$ such that 
$s(\beta_i) = y_i$, where $\beta_i$ are distinct (because the 
$y_i$ are distinct). Now for $i,j \in I$, 
$$ u (\beta_i) = u^{\#} (s (\beta_i)) = u^{\#} (y_i) = 
u^{\#} (y_j) = u^{\#} (s (\beta_j)) = u (\beta_j), $$
from which we deduce the existence of a level set of $u$ with 
uncountable intersection with $\tilde{X}$. This contradicts the 
definition of $\tilde{X}$, and completes the proof. 
\hfill\qedsymbol

\vspace{3mm}

\noindent {\bf Proof of Theorem \ref{thm:2}} 

\noindent Suppose, for a contradiction, that $u$ does 
not have a polar factorisation through $Y$, that is there exists
a measure-preserving mapping $s: X \rightarrow Y$ such that 
$u = u^{\#} \circ s$. Write $X_0 = u^{-1} (B)$, $Y_0 =
(u^{\#})^{-1} (B)$, and denote the restriction of $u, u^{\#}$ 
to $X_0, Y_0$ respectively by $u_0, u_0^{\#}$. Modifying 
$s$ on a set of measure zero if necessary, we have 
$u_0 = u_0^{\#} \circ s$ where $s : X_0 \rightarrow Y_0$ is a 
measure-preserving mapping; moreover $u_0^{\#}$ is the monotone 
rearrangement (through $Y_0$) of $u_0$, so $u_0 = u_0^{\#} 
\circ s$ is a polar factorisation. Now Theorem 1 yields that such 
a polar factorisation cannot exist. This completes the proof. 
\hfill\qedsymbol

\section{Existence of polar inclusions}
\label{s:3}

Theorem 3 will be proved in this section. 
We begin by noting that the concept of polar inclusion is a 
natural extension of polar factorisation. If $u=u^{\#} \circ s$
for some measure-preserving mapping $s : X \rightarrow Y$, we can introduce
a push-forward measure $\pi$ by defining 
\begin{equation}
\label{eq:helen}
\pi (B) = (id \times s)_{\#}\mu (B) \equiv \mu \{ x \in X: (x, s(x)) \in B \}
\end{equation}
for measurable $B \subset X \times Y$. It is easily seen that 
$\pi \in \Pi (\mu, \lambda_n)$, and given $u^{\#} = \nabla \psi$
for some convex $\psi$, 
we have that $u(x) = \partial 
\psi (s(x))$ for $\mu$ a.e. $x \in X$, from which it follows that 
$u(x) \in \partial \psi (y)$ for $\pi$ a.e. $(x,y)$. Moreover the 
Monge-Kantorovich problem of minimising $I (\gamma)$ over 
$\gamma \in \Pi (\mu, \lambda_n)$ can be seen as an extension 
of \cite[Corollary 1]{BD03}: if we restrict to $\pi$ of the form 
(\ref{eq:helen}), the minimisation problem is exactly that of 
finding closest measure-preserving mappings (in $L^2$) to a square 
integrable function $u$, which corresponds to finding measure-preserving 
mappings (if any) which arise as polar factors.

The proof of Theorem 3 makes use of 
some standard ideas from the theory of 
optimal mass transfer problems. We introduce notation for concepts 
from convex analysis and its generalisation to $c$-concave 
analysis. 

\vspace{3mm}

\noindent {\em Notation}. If $\psi : \RR^n \rightarrow \ER$, then 
$\psi^* : \RR^n \rightarrow \ER$ denotes the (Legendre-Fenchel)
{\em conjugate convex function} of $\psi$, defined by 
$$ \psi^* (x) = \sup \{ x \cdot y - \psi (y) : y \in \RR^n \}. $$
Let $X$ and $Y$ be Polish spaces, 
and suppose that $c: X \times Y \rightarrow
[0, \infty) \cup \{+ \infty\}$ is lower semicontinuous. For 
$\phi : Y \rightarrow \RR \cup \{ -\infty \}$, define its 
$c$-transform $\phi^c$ and second $c$-transform $\phi^{cc}$ by 
$$ \phi^c (x) = \inf \{ c(x,y) - \phi (y): y \in Y \}, \ 
\phi^{cc} (y) = \inf \{ c(x,y) - \phi^c (x): x \in X \}. $$
We say that $\phi$ is $c$-{\em concave} if there exists some function 
$\theta$ such that $\phi = \theta^c$. For $c$-concave 
$\phi$ it is easily seen that $\phi^{cc} = \phi$, and   
$(\phi^c, \phi)$ is called a conjugate $c$-concave function pair. 
(Further details may be found in Villani \cite{VILL}, for example.)

\vspace{3mm}

\noindent{\bf Proof of Theorem \ref{thm:3}}

\noindent (i) Standard convex analysis (see for example Rockafellar 
\cite[Theorem 23.5]{RTR}) yields that $u(x) \in \partial \psi (y)$ 
for $\pi$ a.e. $(x,y)$ if and only if  
\begin{equation}
\label{eq:zoe}
\psi^* (u(x)) + \psi (y) - u(x) \cdot y = 0 \mbox{ $\pi$ a.e. 
$(x,y)$}. 
\end{equation}
Now (\ref{eq:zoe}) holds if and only if 
$$ \int_{X \times Y} \psi^* (u(x)) + \psi (y) - u(x) \cdot 
y d\pi (x,y) = 0, $$
noting that the integrand is non-negative. 

Let $\gamma \in \Pi (\mu, \lambda_n)$. Now 
$$ \psi^* (u(x)) + \psi (y) - u(x) \cdot y \geq 0 \mbox{ $\gamma$ a.e. 
$(x,y)$}. $$
It follows that 
$$ J (\gamma) \equiv \int_{X \times Y} \psi^* (u(x)) + 
\psi (y) - u(x) \cdot y d \gamma (x,y) \geq 0 = J(\pi). $$ 
Now 
\begin{eqnarray}
J (\gamma) & = & I (\gamma) + \int_{X \times Y} 
\psi^* (u(x)) + \psi (y) - \frac{|u(x)|^2}{2} -\frac{|y|^2}{2} 
d \gamma (x,y) \nonumber \\
 & = & I (\gamma) + \int_X \psi^* (u(x)) - \frac{|u(x)|^2}{2} d\mu (x) 
+ \int_Y \psi(y) - \frac{|y|^2}{2} d \lambda_n (y), \nonumber 
\end{eqnarray} 
where the second equality follows by noting that $\gamma \in 
\Pi (\mu, \lambda_n)$. It follows that any minimiser of $I$ is a 
minimiser of $J$ (and vice versa); we deduce that 
$\pi$ minimises $I (\gamma)$ over $\gamma \in \Pi (\mu, \lambda_n)$.  

\noindent (ii) Let $\pi$ be a minimiser of $I(\gamma)$ over 
$\Pi(\mu, \lambda_n)$. (At least one minimiser exists by Villani 
\cite[Theorem 1.3]{VILL}.) Now Kantorovich duality 
(see \cite[Theorem 1.3, Remark 2.40]{VILL}) yields that  
\begin{equation}
\label{eq:adam}
\int_{X \times Y} \frac{|u(x) -y|^2}{2} d \pi (x,y) = 
\int_X \phi^c (x) d\mu (x) + \int_Y \phi(y) d \lambda_n (y)
\end{equation}
for a $c$-concave pair $(\phi^c, \phi)$. Now (\ref{eq:adam}) 
implies that 
\begin{equation}
\label{eq:claire}
\int_{X \times Y} \frac{|y|^2}{2} - \phi(y) + \frac{|u(x)|^2}{2}
- \phi^c (x) - u(x) \cdot y d\pi (x,y) = 0. 
\end{equation}
Define $\psi (y) = |y|^2/2 - \phi(y)$. Then it may be shown that 
$\psi^* (u(x)) = |u(x)|^2/2 - \phi^c (x)$, and that 
$\psi$ is the supremum of a family of continuous affine 
functions, whence convex and lower semicontinuous (by 
\cite[Proposition 3.1, p. 14]{ET} for example).   
Moreover the left hand side of (\ref{eq:claire})
is $J (\pi)$. Now the proof of (i) yields that $u(x) \in \partial 
\psi (y)$ for $\pi$ a.e. $(x,y)$. 
\hfill\qedsymbol 

\section*{Acknowledgements}
The author would like to acknowledge stimulating discussions 
with G.R. Burton.

\bibliographystyle{plain}

\end{document}